\documentclass[12pt]{amsart}

\usepackage{amsfonts,amsmath,latexsym,amssymb,verbatim,amsbsy,times}
\usepackage{mathrsfs}
\usepackage{amsthm}
\usepackage{graphicx}
\usepackage{caption}
\usepackage{subcaption}
\usepackage{dsfont}

\usepackage[top=0.8 in, bottom=0.8in, left=0.8in, right=0.8in]{geometry}

\usepackage[colorlinks=true, pdfstartview=FitV, linkcolor=blue,citecolor=green, urlcolor=blue]{hyperref}

\usepackage{enumitem}

\theoremstyle{plain}
\newtheorem{THEOREM}{Theorem}[section]

\newtheorem{LEMMA}[THEOREM]{Lemma}

\theoremstyle{definition}

\theoremstyle{remark}

\newcommand{\N}{\ensuremath{\mathbb{N}}}   
\newcommand{\R}{\ensuremath{\mathbb{R}}}   
\newcommand{\C}{\ensuremath{\mathbb{C}}}   

\def \a {\alpha}
\def \b {\beta}
\def \d {\delta}
\def \g {\gamma}

\def \f {\varphi}

\def \k {\kappa}
\def \l {\lambda}

\def \s {\sigma}

\def \t {\tau}

\def \O {\Omega}

\def \cL {\mathcal{L}}

\def \cR {\mathcal{R}}
\def \cS {\mathcal{S}}

\def \cU {\mathcal{U}}

\def\cprime{$'$}

\def \< {\langle}
\def \> {\rangle}
\def \p {\partial}

\DeclareMathOperator{\im}{Im} %
\DeclareMathOperator{\tr}{Tr} %

\def \dd  {\mathrm{d}}

\def \dx  {\, \mathrm{d}x}

\def \dz  {\, \mathrm{d}z}

\title{$H^s$ Bounds for the Derivative Nonlinear Schr\"odinger Equation}

\author{Hajer Bahouri}
\address[H. Bahouri]{CNRS \& Sorbonne Universit\'e, Laboratoire Jacques-Louis Lions (LJLL) UMR 7598, Place Jussieu, 75005 Paris, France}
\email{hajer.bahouri@ljll.math.upmc.fr}

\author{Trevor M. Leslie}
\address[T.M. Leslie]{University of Southern California, 3620 S. Vermont Ave., Los Angeles, CA 90089}
\email{lesliet@usc.edu}

\author{Galina Perelman} 
\address[G. Perelman]{Laboratoire D'Analyse et de Math\'ematiques Appliqu\'ees UMR 8050, Universit\'e Paris-Est Cr\'eteil, 61, Avenue Du G\'en\'eral De Gaulle, 94010 Cr\'eteil Cedex, France}
\email{galina.perelman@u-pec.fr}

\date{\today}
\thanks{This material is based upon work supported by the National Science Foundation 	under Grant No. DMS-1928930 while the authors participated in a program hosted by the Mathematical Sciences Research Institute in Berkeley, California, during the Spring 2021 semester.}

\begin{document}
	
	\maketitle	
	
\begin{abstract}
We study the derivative nonlinear Schr\"odinger equation on the real line and obtain global-in-time bounds on high order Sobolev norms.
\end{abstract}

\section{Introduction}

We consider the Cauchy problem for the derivative nonlinear Schr\"odinger equation (DNLS) on the real line~$\R$:
\begin{equation}
\label{e:DNLS}
\arraycolsep=1.4pt\def\arraystretch{1.5}
\left\{
\begin{array}{rcl}
i\p_t u + \p_{x}^2 u & = & - i \p_x(|u|^2 u), \\
u\big|_{t=0} & = & u_0\in H^s(\R), \; s\ge \frac12.
\end{array}
\right.
\end{equation}
We remark right away that the DNLS is $L^2$ critical, as it is invariant under the scaling 
\begin{equation}
\label{e:scaling}
u(t,x)\mapsto u_\mu(t,x):=\sqrt{\mu} u(\mu^2 t, \mu x),
\qquad \mu>0.
\end{equation}
The DNLS equation was introduced by Mio-Ogino-Minami-Takeda and Mj\o lhus~\cite{MioOginoMinamiTakeda1976, Mjolhus1976} as a model for studying magnetohydrodynamics, and it has received a great deal of attention from the mathematics community after being shown to be completely integrable by Kaup-Newell~\cite{KaupNewell1978}.  The infinitely many conserved quantities admitted by the DNLS equation play an important role in the wellposedness theory.  The first three---the mass, momentum, and energy---are as follows. 
\begin{align}
\label{e:mass}
M(u) & := \int_\R |u|^2 \dx, \\
P(u) & := \im \int_\R \overline{u} u_x \dx + \frac12 \int_\R |u|^4 \dx, \\
\label{e:E}
E(u) & := \int_\R \big( |u_x|^2 - \frac32 \im(|u|^2 u\overline{u}_x) + \frac12 |u|^6 \big) \dx.
\end{align}

Before stating our main result, let us give a very brief review of what is known about the wellposedness of the DNLS equation. More detailed overviews can be found, for example, in the introductions of~\cite{BahouriPerelman2020} and~\cite{KillipNtekoumeVisan2021}.  Local wellposedness in~$H^s(\R)$ for~$s\ge \frac12$ was proven by Takaoka~\cite{Takaoka1999}, improving earlier work~\cite{Ozawa1996} by Ozawa.  On the other hand, for~$s<\frac12$, the uniform continuity of the data-to-solution map fails in~$H^s(\R)$~\cite{BiagioniLinares2001, Takaoka2001}.  One can, however, close the~$\frac12$-derivative gap between the~$H^\frac12$ threshold and the critical space~$L^2(\R)$ by working in more general Fourier-Lebesgue spaces, c.f. Gr\"unrock~\cite{Grunrock2005} and references therein.  

A line of results, due to Hayashi-Ozawa~\cite{HayashiOzawa1992}, Colliander-Keel-Staffilani-Takaoka-Tao~\cite{CollianderKeelStaffilaniTakaokaTao2002}, Wu~\cite{WuDNLS2015}, and Guo-Wu~\cite{GuoWu2017}, establishes global well-posedness of the DNLS equation in~$H^s(\R)$ for~$s\ge \frac12$, for initial data having mass less than~$4\pi$.  Another line (Pelinovsky-Saalmann-Shimabukuro~\cite{PelinovskySaalmannShimabukuro2017}, Pelinovsky-Shimabukuro~\cite{PelinovskyShimabukuro2018}, and Jenkins-Liu-Perry-Sulem~\cite{JenkinsLiuPerrySulem2018, JenkinsLiuPerrySulem2020APDE, JenkinsLiuPerrySulem2020QJPAM}) uses inverse scattering techniques to establish global wellposedness under stronger regularity and decay assumptions on the initial data, but without a smallness requirement on the mass.  

The first and third authors proved in~\cite{BahouriPerelman2020} that the DNLS equation is globally well-posed in~$H^{s}(\R)$ for~$s\ge \frac12$ and that solutions generated from~$H^{\frac12}$ initial data remain bounded in~$H^{\frac12}(\R)$ for all time.  There have also been some recent works below the aforementioned~$s=\frac12$ threshold of uniform~$H^s$ continuity with respect to initial data~\cite{BiagioniLinares2001, Takaoka2001}.  Klaus-Schippa~\cite{KlausSchippa2020} gave $H^s$ a priori estimates for $0<s<\frac12$ in the case of small mass, Killip-Ntekoume-Vi\c{s}an~\cite{KillipNtekoumeVisan2021} improved the small mass assumption to $4\pi$ and furthermore proved a global wellposedness result in~$H^s(\R)$,~$\frac16\le s<\frac12$, for initial data with mass less than~$4\pi$. Very recently, Harrop-Griffiths, Killip, and Vi\c{s}an~\cite{HarropGriffithsKillipVisan2021} have removed the small mass assumption both from their~$H^s$ a~priori bounds,~$0<s<\frac12$, as well as from their global wellposedness result in~$H^s(\R)$ with~$\frac16\le s<\frac12$.

In this paper, we are concerned with the global-in-time boundedness of solutions to the DNLS equation in~$H^s$ spaces.  We prove that a uniform-in-time bound in~$H^s(\R)$ holds for all~$s \ge  \frac12$.
\begin{THEOREM}
	\label{t:Hs}
	Suppose~$u$ is a solution to the DNLS equation with initial data~$u_0\in H^s(\R)$,~$s \ge \frac12$. There exists a finite positive constant~$C = C(s, \|u_0\|_{H^s(\R)})$, such that 
	\[
	\sup_{t\in \R} \|u(t)\|_{H^s(\R)} \le C(s, \|u_0\|_{H^s(\R)}).
	\]	
\end{THEOREM}

The main idea is to build off of the~$H^{s}$ bounds with $0<s<\frac12$ from~\cite{HarropGriffithsKillipVisan2021} and to take advantage of the complete integrability of the equation.  As in \cite{BahouriPerelman2020}, \cite{HarropGriffithsKillipVisan2021}, the present work relies heavily on the conservation of the transmission coefficient for the spectral problem associated to the DNLS equation.  This property has already been used in many other works; of particular relevance to us are the papers of G\'erard~\cite{GerardIntegrales}, Killip-Vi\c{s}an-Zhang~\cite{KillipVisanZhang2018}, Killip-Vi\c{s}an~\cite{KillipVisan2019}, and Koch-Tataru~\cite{KochTataru2018}, on the cubic NLS and KdV equations.  

Note that by continuity of the flow, and the preservation of the Schwartz class under the flow, we lose nothing by restricting attention to the Schwartz class; we will thus work exclusively with Schwartz functions for the remainder of the manuscript.  We will also suppress the time dependence when it does not play a role.  

One can easily prove Theorem \ref{t:Hs} in the special case $s = 1$, using the conserved quantity $E(u)$.  Indeed, simply rearranging \eqref{e:E} yields
\[
\|u\|_{\dot{H}^1(\R)}^2  = E(u) - \frac12 \|u\|_{L^6(\R)}^6  + \frac32 \int_\R \im(|u|^2 u \overline{u}_x)\dx.
\]
Clearly, the last term can be bounded above in absolute value by $\frac12 \|u\|_{\dot{H}^1(\R)}^2 + C\|u\|_{L^6(\R)}^6$, whence the desired bound follows by Sobolev embedding.

The higher-order Sobolev norms of integer order can be dealt with similarly, once we have a formula for the corresponding higher-order conserved quantities.  We will show that for any nonnegative integer~$\ell$, one of the conserved quantities is equal to a constant multiple of~$\|u\|_{\dot{H}^\ell(\R)}^2$, plus terms which are of lower order.  For noninteger~$s$, we will use a sort of `generalized energy', comparable to~$\|u\|_{\dot{H}^s(\R)}^2$, that will be defined in terms of the transmission coefficient of the DNLS spectral problem.  We sketch presently the background necessary to define these objects precisely; for more details, see, for example,~\cite{AblowitzSegurbook, JenkinsLiuPerrySulem2018, JenkinsLiuPerrySulem2020APDE, JenkinsLiuPerrySulem2020QJPAM, KaupNewell1978,LeeDNLS1989, PelinovskyShimabukuro2018, TsuchidaWadati1999}.

The DNLS equation can be obtained as a compatibility condition of the following system~\cite{KaupNewell1978}:
\begin{equation}
\label{e:KaupNewell}
\begin{split}
\p_x \psi & = \cU(\l) \psi,\\
\p_t \psi & = \Upsilon(\l) \psi.
\end{split}
\end{equation}
Here $\l\in \C$ is a spectral parameter, independent of~$t$ and~$x$, and~$\psi = \psi(t,x,\l)$ is~$\C^2$-valued.  The operators~$\cU(\l)$ and~$\Upsilon(\l)$ are defined by 
\begin{equation}
\begin{split}
\cU(\l) & = -i\s_3 (\l^2 + i\l U), \\
\Upsilon(\l) & = -i(2\l^4 - \l^2 |u|^2)\s_3 + \begin{pmatrix}
0 & 2\l^3 u - \l |u|^2 u + i\l u_x \\
-2\l^3 \overline{u} + \l |u|^2 u + i\l \overline{u}_x & 0 
\end{pmatrix},
\end{split}
\end{equation}
where 
\[
\s_3 = \begin{pmatrix}
1 & 0 \\
0 & -1
\end{pmatrix}, 
\qquad U = \begin{pmatrix}
0 & u \\
\overline{u} & 0 \end{pmatrix}.
\]
To be more specific about the sense in which the DNLS is a compatibility condition, we note that~$u$ satisfies the DNLS equation if and only if~$\cU$ and~$\Upsilon$ satisfy the so-called `zero-curvature' representation 
\[
\frac{\p \cU}{\p t} - \frac{\p \Upsilon}{\p x} + [\cU, \Upsilon] = 0.
\]
The first equation of~\eqref{e:KaupNewell} can be written in the form 
\begin{equation}
\label{e:scattering}
L_u(\l) \psi := (i\s_3 \p_x - \l^2 - i\l U)\psi = 0, 
\end{equation}
which defines the scattering transform associated to the DNLS.  Let us denote 
\[
\O_+:=\{\l\in \C:\im \l^2>0\}.
\]  
Then given~$u\in \cS(\R)$ and~$\l\in \overline{\O}_+$, there are unique solutions to~\eqref{e:scattering} (the ``J\"ost solutions'') exhibiting the following behavior at~$\pm \infty$: 
\begin{equation}
\begin{split}
\psi_1^-(x,\l) = e^{-i\l^2 x}  \left[ \begin{pmatrix}
1 \\ 0 
\end{pmatrix}
+ o(1) \right], & \qquad \text{ as } x\to -\infty, \\
\psi_2^+(x,\l) = e^{i\l^2 x}  \left[ \begin{pmatrix}
0 \\ 1 
\end{pmatrix}
+ o(1) \right], & \qquad \text{ as } x\to +\infty.
\end{split}
\end{equation}
Finally, we denote by~$a_u(\l)$ the Wronskian of the J\"ost solutions defined above:\footnote{The \textit{transmission coefficient} mentioned earlier is the inverse of $a_u(\l)$.}
\begin{equation}
a_u(\l) = \det(\psi_1^-(x,\l), \psi_2^+(x,\l)).
\end{equation}

Using the second equation in~\eqref{e:KaupNewell}, it can be shown that~$a_u(\l)$ is time-independent if~$u$ is a solution of~\eqref{e:DNLS}.  Furthermore, $a_u$ is a holomorphic function of $\l$ in $\O_+$, and one may determine the behavior of~$a_u$ at infinity by transforming~\eqref{e:scattering} into a Zakharov-Shabat spectral problem, linear with respect to the spectral parameter, c.f. \cite{KaupNewell1978}, \cite{PelinovskySaalmannShimabukuro2017}.  The equivalence between the two problems allows us to write 
\begin{equation}
\label{e:auatinfinity}
\lim_{|\l|\to \infty,\,\l\in \overline{\O}_+} a_u(\l) = e^{-\frac{i}{2} \|u\|_{L^2(\R)}^2}.
\end{equation}
For fixed~$u$, we can thus define the logarithm so that 
\begin{equation}
\label{e:lndef}
\lim_{|\l|\to \infty, \l \in \overline{\O}_+} \ln a_u(\l) = -\frac{i}{2}\|u\|_{L^2(\R)}^2.
\end{equation}
Moreover,~$\ln a_u(\l)$ admits an asymptotic expansion of the following form:
\begin{equation}
\label{e:energyexpansion}
\ln a_u(\l) = \sum_{j=0}^\infty \frac{E_j(u)}{\l^{2j}}
\qquad 
\text{ as } |\l|\to \infty,\; \l\in \overline{\O}_+.
\end{equation}
Since~$a_u(\l)$ is time-independent, the quantities~$E_j(u)$ are conservation laws. They are all polynomial in~$u$, $\overline{u}$, and their derivatives.  Furthermore, the~$E_j(u)$'s inherit scaling properties from $a_u(\l)$.  That is, for~$\mu>0$, the fact that $a_{u_\mu}(\l) = a_u(\frac{\l}{\sqrt{\mu}})$ implies that $E_j(u_\mu) = \mu^j E_j(u)$, for each $j\in \N$.  The first several of the $E_j(u)$'s are (up to multiplicative constants) the conserved quantities~\eqref{e:mass}--\eqref{e:E} mentioned earlier: 
\[
E_0(u) = -\frac{i}{2}\|u\|_{L^2(\R)}^2 = -\frac{i}{2} M(u), 
\qquad 
E_1(u) = \frac{i}{4} P(u), 
\qquad E_2(u) = -\frac{i}{8}E(u).
\]
For each~$\ell\in \N^*$, the quantity~$E_{2\ell}(u)$ can be used to control~$\|u\|_{\dot{H}^\ell(\R)}^2$.  Let us define, for $\rho$ positive sufficiently large and~$L\in \N$, 
\begin{equation}
\f_L(u,\rho) = \im\left[ \ln a_u(\sqrt{i\rho}) - \sum_{j=0}^{2L+1} \frac{E_j(u)}{(i\rho)^j} \right] .
\end{equation}
If $u$ is a solution of the DNLS equation, then $\f_L(u,\rho)$ is time-independent, being a sum of time-independent quantities.

In order to establish bounds on the~$H^s$ norm of~$u$, for~$s\ge \frac12$, we will show that~$\int_{R}^\infty \rho^{2s-1} \f_{[s]}(u,\rho) \dd\rho$ with $R>0$ large enough controls the~$\dot{H}^s$ seminorm of~$u$, in a sense to be made precise later.  Here and below, we use~$[s]$ to denote the integer part of a real number~$s$.

Our proof of Theorem~\ref{t:Hs} relies on a good understanding of the structure of the remainder associated to the expansion~\eqref{e:energyexpansion}.  Note that when~$\l^2 = i\rho$, the imaginary part of this remainder (which is what we really use) is simply~$\f_L(u, \rho)$.  In Section~\ref{s:mainpfs}, we will introduce a determinant characterization of~$a_u(\l)$; we use this characterization to formulate a technical statement (Lemma~\ref{l:trdecomp} below) on the size of the remainder.  Assuming the result of Lemma~\ref{l:trdecomp}, we will prove Theorem~\ref{t:Hs} at the end of Section~\ref{s:mainpfs}.  Then, in Section~\ref{s:decompproof}, we will prove our technical Lemma, completing the circle of ideas.  Most of the work is contained in this last section.

Before moving on, let us establish a few notational conventions that we wish to add to the ones introduced above.  First of all, we use the following normalization for the Fourier transform: 
\[
\widehat{f}(\zeta) = \frac{1}{\sqrt{2\pi}} \int_\R e^{ix\zeta} f(x)\dx.
\]
The symbol~$\N$ will denote the nonnegative integers, and~$\N^* = \N\backslash \{0\}$.  We will use~$\|\cdot\|_2$ to denote the Hilbert-Schmidt norm, and~$\|\cdot\|$ will denote the operator norm on~$L^2(\R)$. And we will use the following shorthand for derivatives:
\[
D = -i\p_x, \qquad \cL_0 = i\s_3 \p_x.
\]
Whenever $2\le p<\infty$, we will use $s^*(p)$ to denote the Sobolev exponent $s^*(p) = \frac12 - \frac1p$ such that the embedding $H^{s^*(p)}(\R)\hookrightarrow L^p(\R)$ holds.  

Finally, we set notation for the following subset of $\O_+$:
\[
\Gamma_\d = \{\l\in \O_+:\d<\arg(\l^2)<\pi - \d\}.
\]
This notation will be useful in some of the intermediate steps we use to prove Theorem~\ref{t:Hs}, as our estimates will frequently depend on~$\frac{|\l|^2}{\im \l^2}$ (which is~$\le C(\d)$ on~$\Gamma_\d$).  However, the value of~$\d>0$ will be inconsequential for our final steps, where we will take~$\l^2$ to be pure imaginary.  Therefore, for simplicity of presentation, we will \textit{fix}~$\d>0$ once and for all and suppress dependence on~$\d$ in all bounds below.

\section{Proof of the Main Result}

\label{s:mainpfs}

\subsection{The Determinant Characterization of $a_u(\l)$}

\label{ss:detcharau}

An important property of~$a_u(\l)$ is the fact that it can be realized as a perturbation determinant:
\begin{equation}
\label{e:audet}
a_u(\l)^2 = \det(I - T_u(\l)^2), 
\end{equation}
where 
\[
T_u(\l) = i\l (\cL_0 - \l^2)^{-1}U, \qquad \l\in \O_+.
\]
The operator~$T_u(\l)$ is Hilbert-Schmidt, with 
\begin{equation}
\|T_u(\l)\|_2^2 = \frac{|\l|^2}{\im(\l^2)} \|u\|_{L^2(\R)}^2.
\end{equation}
As a consequence of~\eqref{e:audet}, we may write\footnote{This series expansion of $\ln a_u(\l)$ is consistent with the definition \eqref{e:lndef}.}
\begin{equation}
\label{e:trexpansion}
\ln a_u(\l) = -\sum_{k=1}^\infty \frac{\tr (T_u(\l)^{2k})}{2k}, \qquad \text{ if } \|T_u(\l)\|<1.
\end{equation}
This series will converge whenever~$\l\in \Gamma_\d$ has large enough modulus; indeed, using the explicit kernel of~$(\cL_0 - \l^2)^{-1}$, it can easily be shown that for any~$p> 2$, we have
\begin{equation}
\label{e:Tuop}
\|T_u(\l)\| \lesssim \frac{|\l| \|u\|_{L^p(\R)}}{\im (\l^2)^{1-\frac1p}},
\qquad \l\in \O_+, \; u\in L^p(\R). 
\end{equation}
In particular, we can find~$R_0 = R_0(\|u\|_{H^{\frac13}(\R)})$ such that~$\|T_u(\l)\|\le \frac12$ for all~$\l\in \Gamma_\d$ satisfying~$|\l|^2\ge R_0$. We will fix the notation~$R_0$ for use below.  

As we shall see later, each term of the series~\eqref{e:trexpansion} can be expanded in powers of~$\l^{-2}$:
\begin{equation}
\label{e:Trmu}
-\frac{\tr (T_u(\l)^{2k})}{2k} = \sum_{j=k-1}^\infty \frac{\mu_{j,k}(u)}{\l^{2j}}.
\end{equation}
According to \eqref{e:energyexpansion} and \eqref{e:trexpansion}, the~$E_j(u)$'s should then satisfy 
\begin{equation}
\label{e:Ejmu}
E_j(u) = \sum_{k=1}^{j+1} \mu_{j,k}(u).
\end{equation}
We will use the following notation for the remainders after truncation of the expansions~\eqref{e:trexpansion} and~\eqref{e:Trmu}:
\begin{equation}
\ln a_u(\l) = -\sum_{k=1}^{2L+2} \frac{\tr (T_u(\l)^{2k})}{2k} + \t_L^*(u,\l),
\qquad L\in \N;
\end{equation}
\begin{equation}
\label{e:trdecomp}
-\frac{\tr(T_u(\l)^{2k})}{2k} = \sum_{j=k-1}^{2L+1} \frac{\mu_{j,k}(u)}{\l^{2j}} + \t^{k}_{L}(u,\l),
\qquad k\in \{1, \ldots, 2L+2\}, \; L\in \N. 
\end{equation}	
The primary difficulty of the proof of Theorem~\ref{t:Hs}---and indeed, the subject of Lemma~\ref{l:trdecomp}---is the understanding of the size and structure of the remainder terms~$\t^k_L(u,\l)$, and to a lesser extent, the~$\mu_{j,k}(u)$'s. On the other hand, for~$\l\in \Gamma_\d$ with large enough modulus, it is easy to bound the~$\t_L^*(u,\l)$'s.  For example, if~$\|T_u(\l)\|\le \frac12$, then 
\begin{equation}
\label{e:tau*}
\begin{split}
|\t_L^*(u,\l)| & = \bigg| \ln a_u(\l) + \sum_{k=1}^{2L+2} \frac{\tr T_u^{2k}(\l)}{2k} \bigg| \le \sum_{k=2L+3}^\infty \| T_u(\l)\|^{2k-2} \|T_u(\l)\|_2^2 \\
& \lesssim \|T_u(\l)\|^{4L+4} \|T_u(\l)\|_2^2 \lesssim \frac{\|u\|_{H^{s^*(p)}(\R)}^{4L+4} \|u\|_{L^2(\R)}^2}{|\l|^{(4L+4)(1 - \frac2p)}}, \quad 2< p<\infty, \; s^*(p) = \frac12 - \frac1p.
\end{split}
\end{equation}

The following table summarizes the various relationships among the quantities introduced above and will be helpful to keep track of the numerology.  More precise information about the~$\mu_{j,k}(u)$'s and~$\t_L^k(u,\l)$'s will be provided below.

\begin{tiny}
	\[	
	\renewcommand*{\arraystretch}{2.5}
	\hspace*{-7mm}\begin{array}{rcccccccccccccc}
	& & -\dfrac{\tr T_u^2(\l)}{2} & & -\dfrac{\tr T_u^4(\l)}{4} & & -\dfrac{\tr T_u^6(\l)}{6} & & \cdots & & -\dfrac{\tr T_u^{4L+2}(\l)}{4L+2} & & -\dfrac{\tr T_u^{4L+4}(\l)}{4L+4} &  &  \\ \cline{3-14}
	\ln a_u(\l) = & & \mu_{0,1}(u) &&& &&&&&&&&\vline & E_0(u)\\ [0.1cm]
	& +&  \dfrac{\mu_{1,1}(u)}{\l^2} &+&  \dfrac{\mu_{1,2}(u)}{\l^2}  &&&&&&&&&\vline & \dfrac{E_1(u)}{\l^2} \\
	& +& \dfrac{\mu_{2,1}(u)}{\l^4} &+&  \dfrac{\mu_{2,2}(u)}{\l^4} &+& \dfrac{\mu_{2,3}(u)}{\l^4} &&&&&&&\vline & \dfrac{E_2(u)}{\l^4}  \\
	& & \vdots && \vdots && \vdots &  & \ddots &&&&&\vline & \vdots \\
	& +& \dfrac{\mu_{2L,1}(u)}{\l^{4L}} &+&  \dfrac{\mu_{2L,2}(u)}{\l^{4L}} &+& \dfrac{\mu_{2L,3}(u)}{\l^{4L}} & + & \cdots & + & \dfrac{\mu_{2L,2L+1}(u)}{\l^{4L}} &&&\vline & \dfrac{E_{2L}(u)}{\l^{4L}}  \\ 
	& +& \dfrac{\mu_{2L+1,1}(u)}{\l^{4L+2}} &+&  \dfrac{\mu_{2L+1,2}(u)}{\l^{4L+2}} &+& \dfrac{\mu_{2L+1,3}(u)}{\l^{4L+2}} & + & \cdots & + & \dfrac{\mu_{2L+1,2L+1}(u)}{\l^{4L+2}}  & + & \dfrac{\mu_{2L+1,2L+2}(u)}{\l^{4L+2}} &\vline & \dfrac{E_{2L+1}(u)}{\l^{4L+2}}  \\[.3cm] \cline{14-15} 
	& + & \t_{L}^1(u,\l) & + & \t_{L}^2(u,\l) & + & \t_{L}^3(u,\l) & + & \cdots & +&  \t_{L}^{2L+1}(u,\l) & +&  \t_{L}^{2L+2}(u,\l) & + & \t_L^*(u,\l)
	\end{array}	
	\]	
\end{tiny}

\subsection{Structure of the Traces}

\label{ss:Lemma}

In this section, we record all the information about the traces that we need in order to prove our main result.  We deal first with the easy case of~$\tr T_u(\l)^2$, about which we need more explicit information.  A straightforward computation gives us 
\begin{equation}
\label{e:trtu2}
\tr T_u^2(\l) = 2i\l^2 \int_\R \frac{|\widehat{u}(\zeta)|^2}{\zeta + 2\l^2} \dd\zeta.
\end{equation}
We determine the expansion of~$\tr T_u^2(\l)$ by simply substituting into~\eqref{e:trtu2} the identity
\[
\frac{2\l^2}{\zeta+2\l^2} = \sum_{j=0}^{2L+1} \left( - \frac{\zeta}{2\l^2} \right)^j + \frac{\zeta}{\zeta+2\l^2} \left( \frac{\zeta}{2\l^2} \right)^{2L+1}, 
\qquad L\in \N,
\]
to obtain
\begin{equation}
\label{e:trtu2expanded}
-\frac{\tr T_u^2(\l)}{2} = \sum_{j=0}^{2L+1} \frac{1}{\l^{2j}} \cdot \underbrace{\frac{i}{(-2)^{j+1}} \int_\R \zeta^j |\widehat{u}(\zeta)|^2 \dd\zeta}_{=:\mu_{j,1}(u)} - \underbrace{\frac{i}{4^{L+1} \l^{4L+2}} \int_\R \frac{\zeta^{2L+2}|\widehat{u}(\zeta)|^2}{\zeta + 2\l^2} \dd\zeta}_{=: \t_L^1(u,\l)}, \qquad L\in \N.
\end{equation}

Now we state our main Lemma, which describes the structure of the other~$\mu_{j,k}(u)$'s and~$\t_L^k(u,\l)$'s.

\begin{LEMMA}
	\label{l:trdecomp}
	For any~$k\in \N^*$,~$L\in \N$, the traces~$\tr T_u^{2k}(\l)$ admit the decomposition~\eqref{e:trdecomp}.  The~$\mu_{j,k}(u)$'s and~$\t_L^k(u,\l)$'s satisfy the properties below, where for any~$n\in \N$ we denote~$\s(n) = \max\{n, \frac13\}$.  
	
	\medskip
	\begin{itemize}
		\item Each~$\mu_{j,k}(u)$ is a homogeneous polynomial of degree~$2k$ in~$u$,~$\overline{u}$, and their derivatives; it is homogeneous with respect to the natural scaling. We have
		\begin{equation}
		\label{e:mu2lbd}
		\begin{split}
		|\mu_{2\ell,2}(u)| & \lesssim \|u\|_{H^{\s(\ell-1)}(\R)}^{3}\|u\|_{H^{\ell}(\R)}, \qquad \ell\in \N^*, \\
		|\mu_{2\ell,k}(u)| & \lesssim \|u\|_{H^{\s(\ell-1)}(\R)}^{2k}, 
		\qquad \qquad \quad \;\; \ell\in \N^*,\;k\in \{3,\ldots, 2\ell + 1\},  \\
		|\mu_{2\ell+1,k}(u)| & \lesssim \|u\|_{H^\ell(\R)}^{2k}, \qquad \qquad \qquad \quad\; \ell\in \N^*, \;k\in \{2, \ldots, 2\ell + 2\}.
		\end{split}
		\end{equation}
		\item For~$|\l|^2>R_0$,~$\l\in \Gamma_\d$, we have the following bounds:
		\begin{align}
		\label{e:tau2bound}
		|\t_L^{2}(u,\l)| & \lesssim_{\a} \frac{\|u\|_{H^{\s(L)}(\R)}^{3}\|u\|_{H^{L+\a}(\R)}}{|\l|^{4L+2 + 2\a}},  & L\in \N,\; 0\le \a<1;\\
		\label{e:taubound}
		|\t_L^{k}(u,\l)| & \lesssim \frac{\|u\|_{H^{L}(\R)}^{2k}}{|\l|^{4L+4}}, 
		& L\in \N^*, \; k\in \{3, \ldots, 2L+2\}. 
		\end{align}
	\end{itemize}
\end{LEMMA}	

We postpone the proof of the Lemma until Section~\ref{s:decompproof}.  

\subsection{Proof of Theorem \ref{t:Hs}}

\label{s:mainproof}

In this section, we will prove Theorem~\ref{t:Hs}, assuming the result of Lemma~\ref{l:trdecomp}.  For~$s\in \N^*$, the conclusion follows easily from Lemma~\ref{l:trdecomp}, together with~\eqref{e:Ejmu},~\eqref{e:trtu2expanded},  and an induction argument; we provide the details presently.  Actually, the case $s = 1$ was already proved in the Introduction.  Therefore, let us turn to our inductive hypothesis.  For $k=1, \ldots, \ell-1$, we assume that the following bound holds.
\begin{equation}
\label{e:induction}
\sup_{ t\in \R}\|u(t)\|_{H^{k}(\R)} \le C(k,\|u_0\|_{H^{k}(\R)}). 
\end{equation}
We will prove that the same bound holds with $k=\ell\ge 2$.  

First of all, for any integer~$\ell\ge 2$, and any time~$t$, we have 
\begin{align*}
\|u(t)\|_{\dot{H}^\ell(\R)}^2  & = C(\ell) \mu_{2\ell,1}(u(t)) & \text{ by } \eqref{e:trtu2expanded} \\ 
& = C(\ell) \left[ E_{2\ell}(u(t)) - \sum_{k=2}^{2\ell+1} \mu_{2\ell, k}(u(t)) \right]  & \text{ by } \eqref{e:Ejmu} \\
& \le C(\ell) E_{2\ell}(u_0) + \frac{1}{2} \|u(t)\|_{\dot{H}^\ell(\R)}^2 + C(\ell,\|u_0\|_{H^{\ell-1}(\R)}).
\end{align*}
To pass to the last line, we used time-independence of $E_{2\ell}(u(t))$, the bounds \eqref{e:mu2lbd}, and our inductive hypothesis~\eqref{e:induction} (with~$k=\ell-1$). Finally, using that 
\[
E_{2\ell}(u_0) = \sum_{k=1}^{2\ell+1} \mu_{2\ell, k}(u_0) \le C(\ell, \|u_0\|_{H^{\ell}(\R)}),
\]
we get 
\begin{equation*}
\sup_{t\in \R} \|u(t)\|_{\dot{H}^\ell(\R)} \le C(\ell,  \|u_0\|_{H^{\ell}(\R)}),
\end{equation*}
which finishes the induction argument, and thus the proof of Theorem~\ref{t:Hs} for~$s\in \N^*$.  

It remains to consider the situation where~$s\notin \N^*$.  We start by recording the characterization of~$\f_{L}(u,\rho)$ in terms of the remainders~$\t_L^k(u,\sqrt{i\rho})$, and we also set notation for the quadratic part of~$\f_L(u,\rho)$.    We also note that the case~$L=0$ is included in the definition.
\begin{align}
\label{e:fLdef}
\f_L(u,\rho) & = \im\left[ \ln a_u(\sqrt{i\rho}) - \sum_{j=0}^{2L+1}\frac{E_j(u)}{(i\rho)^j} \right] = \im \left[ \sum_{k=1}^{2L+2} \t_L^k(u,\sqrt{i\rho}) + \t_L^*(u,\sqrt{i\rho}) \right], & L\in \N,\\
\label{e:fL0def}
\f_{L,0}(u,\rho) & = \im \t^1_L(u,\sqrt{i\rho}) =  \frac{(-1)^L}{2^{2L+1}\rho^{2L}} \int_\R \frac{\zeta^{2L+2} |\widehat{u}(\zeta)|^2}{\zeta^2 + 4\rho^2} \dd\zeta, & L\in \N.
\end{align}

The conclusion of Theorem~\ref{t:Hs} for noninteger $s\geq\frac12$ will be deduced from the following two Lemmas.
\begin{LEMMA}
	\label{l:Hscomparison}
	Suppose~$u\in \cS(\R)$,~$s>0$,~$s\notin \N^*$, and~$R>0$.  Then the following comparison holds. 
	\begin{equation}
	\label{e:Hstau1[s]comparison}
	\int_{\R_+} \rho^{2s-1} |\f_{[s],0}(u, \rho)| \dd\rho \lesssim_s \|u\|_{\dot{H}^s(\R)}^2 \lesssim_s \int_R^\infty \rho^{2s-1} |\f_{[s],0}(u, \rho)| \dd\rho + R^{2(s-[s])}\|u\|_{\dot{H}^{[s]}(\R)}^2.
	\end{equation}
\end{LEMMA}
\begin{proof}
	Let us define the function~$f_\nu:\R\to \R$, for~$0<\nu< 1$, by~$f_\nu(z) = \frac{|z|^{2\nu - 1}}{1 + z^2}$.	Note that~$f_\nu\in L^1(\R)$ for this range of~$\nu$.  
	
	We make a direct substitution of the formula~\eqref{e:fL0def} for~$\f_{[s],0}(u, \rho)$ into the left side of~\eqref{e:Hstau1[s]comparison}, then we switch the order of integration. Continuing the computation yields 	
	\begin{align*}
	\int_R^\infty \rho^{2s-1} |\f_{[s],0}(u, \rho)|\dd\rho 
	& = \frac{1}{2^{2[s]+1}} \int_\R \zeta^{2[s]+2}|\widehat{u}(\zeta)|^2 \int_R^\infty  \frac{\rho^{2(s-[s])-1}}{\zeta^2 + 4\rho^2} \dd\rho\, \dd\zeta \\
	& = \frac{1}{2^{2s+1}} \int_\R |\zeta|^{2s}|\widehat{u}(\zeta)|^2 \int_{\frac{2R}{|\zeta|}}^\infty f_{s-[s]}(z)\dz\,  \dd\zeta \\
	& =\frac{1}{4^{s+1}} \|f_{s-[s]}\|_{L^1(\R)} \|u\|_{\dot{H}^s(\R)}^2 - \frac12\int_\R \left| \frac{\zeta}{2}\right|^{2s} |\widehat{u}(\zeta)|^2 \int_{0}^{\frac{2R}{|\zeta|}} f_{s-[s]}(z)\dz\, \dd\zeta.
	\end{align*}
	We estimate the second term on the right by means of  the trivial replacement~$\frac{1}{1+z^2} \le 1$:
	\begin{align*}
	\frac12\int_\R \left| \frac{\zeta}{2}\right|^{2s} |\widehat{u}(\zeta)|^2 \int_{0}^{\frac{2R}{|\zeta|}} f_{s-[s]}(z)\dz \,\dd\zeta
	& \le \frac12 \int_\R \left| \frac{\zeta}{2}\right|^{2s} |\widehat{u}(\zeta)|^2 \int_{0}^{\frac{2R}{|\zeta|}} z^{2(s-[s])-1}\dz\, \dd\zeta =  \frac{R^{2(s-[s])}}{s-[s]}\cdot  \frac{\|u\|_{\dot{H}^{[s]}(\R)}^2}{4^{[s]+1}}.
	\end{align*}
	The comparison~\eqref{e:Hstau1[s]comparison} follows.
\end{proof}

\begin{LEMMA}
	\label{l:phiLphiL0}
	Suppose~$u\in \cS(\R)$,~$s>0$,~$s\notin \N^*$. Denoting $\beta = \max\{[s],\, \frac{s+[s]+1}{4([s]+1)},\,\frac13\}$, we have
	\begin{equation}
	\label{e:phiLphiL0}
	|\f_{[s]}(u, \rho) - \f_{[s],0}(u, \rho)|\leq \frac{C(s, \|u\|_{H^{\beta}(\R)})}{\rho^{s+[s] + 1}}(\|u\|_{H^s(\R)}+1), \quad \forall \rho\geq R_0.
	\end{equation}
	
\end{LEMMA}

\begin{proof}
	Choose~$p>2$ to solve~$2([s]+1)(1-\frac2p) = s+[s]+1$.  (Note that $s^*(p) = \frac{s+[s]+1}{4([s]+1)}$ for this choice of $p$.) Then for~$\rho >R_0$, we have
	\begin{align*}
	& |\f_{[s]}(u, \rho) - \f_{[s],0}(u, \rho)|
	\le \sum_{k=2}^{2[s]+2}
	|\t_{[s]}^k(u,\sqrt{i\rho})| + |\t^*_{[s]}(u,\sqrt{i\rho})| & \text{ by } \eqref{e:fLdef}, \eqref{e:fL0def}\\
	&\le C(s) \bigg[ \frac{\|u\|_{H^{\beta}(\R)}^3 \|u\|_{H^s(\R)}}{\rho^{s + [s] + 1}} + \sum_{k=3}^{2[s]+2} \frac{\|u\|_{H^{[s]}(\R)}^{2k}}{\rho^{2[s]+2}} + \frac{\|u\|_{H^{s^*(p)}(\R)}^{4[s]+4} \|u\|_{L^2(\R)}^2}{\rho^{(2[s]+2)(1-\frac2p)}} \bigg] & \text{ by } \eqref{e:tau2bound}, \eqref{e:taubound}, \eqref{e:tau*}\\
	&\le\frac{C(s, \|u\|_{H^{\beta}(\R)})}{\rho^{s+[s]+1}}(\|u\|_{H^s(\R)}+1).
	\end{align*}
	In the second line, we understand the sum over $k$ to be empty if $[s]=0$.	
\end{proof}

The conclusion of Theorem~\ref{t:Hs} for noninteger $s\geq\frac12$ follows from Lemmas~\ref{l:Hscomparison} and \ref{l:phiLphiL0}, the time-independence of the quantity~$\f_{[s]}(u,\rho)$ for solutions of the DNLS equation, and the bound
\begin{equation}
\label{e:beta}
\sup_{t\in \R} \|u(t)\|_{H^\b(\R)} \le C(\b, \|u_0\|_{H^\b(\R)}),
\end{equation}
where $\b=\max\{[s], \frac{s+[s]+1}{4([s]+1)}, \frac13\}$ is as in the statement of Lemma~\ref{l:phiLphiL0}.  The bound~\eqref{e:beta} follows from our induction argument if $s>1$ and from the result of Harrop-Griffiths, Killip, and Vi\c{s}an \cite{HarropGriffithsKillipVisan2021} if $\frac12\le s<1$.

Let us give the remaining details of the proof of Theorem~\ref{t:Hs} presently.  For any $t\in \R$, we have 
\begin{align*}
\|u(t)\|_{\dot{H}^s(\R)}^2 
&\lesssim_s\int_{R_0}^\infty \rho^{2s-1} |\f_{[s],0}(u(t),\rho)| \dd\rho + R_0^{2(s-[s])}\|u(t)\|_{H^{[s]}(\R)}^2 \\
& \le C(s)\int_{R_0}^\infty \rho^{2s-1} |\f_{[s]}(u(t),\rho)| \dd\rho + C(s,R_0, \|u(t)\|_{H^{\b}(\R)})( \|u(t)\|_{H^s(\R)}+1) \\
& \le C(s)\int_{R_0}^\infty \rho^{2s-1} |\f_{[s]}(u_0,\rho)| \dd\rho + C(s,\|u_0\|_{H^{s}(\R)})( \|u(t)\|_{H^s(\R)}+1)\\
& \le  \frac12 \|u(t)\|^2_{H^s(\R)}+C(s,\|u_0\|_{H^{s}(\R)}),
\end{align*}
which establishes the desired conclusion. Note that the first line in the calculation above is simply the upper bound in Lemma~\ref{l:Hscomparison}. To pass from the first line to the second, we use Lemma~\ref{l:phiLphiL0}, followed by the lower bound of Lemma~\ref{l:Hscomparison}.  We use \eqref{e:beta} and the time independence of $\f_{[s]}(u(t), \rho)$ to pass to the third line.  Finally, we justify the last line by noting that 
\[
\int_{R_0}^\infty \rho^{2s-1} |\f_{[s]}(u_0,\rho)| \dd\rho \lesssim_s C(s,\|u_0\|_{H^{s}(\R)}),
\]
which follows from an application of Lemma~\ref{l:phiLphiL0}, followed by the lower bound in Lemma~\ref{l:Hscomparison}.

\section{Proof of Lemma \ref{l:trdecomp}}

\label{s:decompproof}

\subsection{Outline of the Proof}

In this section, we expand each~$\tr(T_u^{2k}(\l))$ in powers of~$\l^{-2}$, up to a specified order, and we establish bounds on the remainders, in order to prove our key Lemma~\ref{l:trdecomp}.  In Section~\ref{ss:L0}, we consider the case~$L=0$, which is easy to treat explicitly but does not fit naturally into our argument for the other cases.  When~$L\ge 1$, we follow the strategy of~\cite{GerardIntegrales}, deducing the expansions of the traces from the expansion of the resolvent~$L_u(\l)^{-1}$. The relationship between~$T_u(\l)$ and~$L_u(\l)$ is the following:
\begin{equation}
L_u(\l) = (\cL_0 - \l^2)(I - T_u(\l)).
\end{equation}
Therefore, 
\begin{equation}
\label{e:Luinvexp}
L_u(\l)^{-1} = (I - T_u(\l))^{-1} (\cL_0 - \l^2)^{-1} = \sum_{n=0}^\infty \underbrace{T_u(\l)^{n} (\cL_0 - \l^2)^{-1}}_{=:\cR_n},\qquad \|T_u(\l)\|<1.
\end{equation}
The point is that 
\begin{equation}
\label{e:TuversusR}
T_u^{2k}(\l) = i\l \cR_{2k-1} U.
\end{equation}
Thus, the part of~$L_u^{-1}(\l)$ that is of relevance to us is~$\cR_{2k-1}$, i.e., the term in the expansion \eqref{e:Luinvexp} that is homogeneous of degree~$2k-1$ in~$u, \overline{u}$.  In particular, we seek an expansion of~$\l\cR_{2k-1}$ in powers of~$\l^{-2}$, up to order~$\l^{4L+2}$ for a given~$L\in \N^*$, and a good understanding of the remainder term.  

Our strategy will be to examine the symbol $R(x,\zeta)$ of the pseudodifferential operator $L_u(\l)^{-1}$.  In Section~\ref{ss:ExpandRes}, we will expand the diagonal and antidiagonal parts $R^d(x,\zeta)$ and $R^a(x,\zeta)$ of $R(x,\zeta)$ in powers of $\l^{-2}$, determining recursively the form of each term of the expansion.  Homogeneity considerations will then give us the desired  expansion of $\l\cR_{2k-1}$ (and thus of $\tr T_u^{2k}(\l)$) in powers of~$\l^{-2}$.  In Section~\ref{ss:extractmuj}, we identify the~$\mu_{j,k}(u)$'s from~\eqref{e:trdecomp} and separate them from the remainder term.  In Section~\ref{ss:estrem} we estimate the remainder term, finishing the proof of the Lemma.  The final Section~\ref{ss:induction} consists of the proof by induction of a technical result stated in Section~\ref{sss:RaRdformalexp}, on the form of the terms of the expansions for~$R^d$ and~$R^a$.

\subsection{Case $L=0$}

\label{ss:L0}

Let us note first of all that the desired decomposition in the case~$L=0$ reads 
\[
\ln a_u(\l) = [\underbrace{\mu_{0,1}(u) + \l^{-2} \mu_{1,1}(u) + \t_0^1(u,\l)}_{=-\frac12 \tr T_u^2(\l)}] + [\underbrace{\l^{-2} \mu_{1,2}(u) + \t_0^2(u,\l)}_{=-\frac14 \tr T_u^4(\l)}] + \t_0^*(u,\l).
\]
(See the table in Section~\ref{ss:detcharau}.) The only term which we have not already understood is~$\t_0^2(u,\l)$; in order to treat it, we decompose~$T_u^4(\l)$ explicitly as follows.  A computation (the details of which are contained, for instance, in~\cite{BahouriPerelman2020}) tells us that 
\[
\tr T_u^4(\l)
= i(2\l^2)^2 \int_\R \overline{u}(x) \big( (D+2\l^2)^{-1} u(x)\big)^2 (D-2\l^2)^{-1} \overline{u}(x) \dx.
\] 
Then, making a few simple manipulations, we can bring the right side of the equation above into the following form. 
\begin{small}
	\begin{align*}
	\tr T_u^4(\l) & = \frac{i}{-2\l^2} \int_\R \overline{u}(x) \bigg[ u(x) -  (D+2\l^2)^{-1} Du(x)\bigg]^2 \big[ \overline{u}(x) - (D-2\l^2)^{-1} D\overline{u}(x) \big] \dx \\
	& = -\frac{i}{2\l^2} \bigg[ \int_\R |u(x)|^4\dx - \int_\R |u|^2 u(x) (D-2\l^2)^{-1} D\overline{u}(x)\dx - 2\int_\R |u|^2 \overline{u}(x)(D+2\l^2)^{-1}Du(x)\dx \\
	& \qquad\qquad  + 2\int_\R |u(x)|^2 (D+2\l^2)^{-1}Du(x) (D-2\l^2)^{-1}D\overline{u}(x)\dx + \int_\R ((D+2\l^2)^{-1}Du(x))^2 \overline{u}(x)^2\dx  \\
	& \qquad\qquad - \int_\R \overline{u}(x) ((D + 2\l^2)^{-1}Du(x))^2 ((D - 2\l^2)^{-1}D\overline{u}(x))\dx \bigg] \\
	& = -\frac{4}{\l^2} \underbrace{\left[ \frac{i}{8} \|u\|_{L^4(\R)}^4 \right]}_{=\mu_{1,2}(u)} - 4\t_0^2(u,\l).
	\end{align*}
\end{small}

To estimate~$\t_0^2(u,\l)$, we use the following simple Lemma, the proof of which we omit.
\begin{LEMMA}
	\label{l:DinvDu}
	The following estimates hold, for~$\l\in \Gamma_\d$.
	\begin{itemize}
		\item If $0\le \a_1 \le \a_2\le 1$, then
		\begin{equation}
		\left\| (D\pm 2\l^2)^{-1} Du \right\|_{\dot{H}^{\a_1}(\R)} \lesssim_{\a_2-\a_1} \frac{ \|u\|_{\dot{H}^{\a_2}(\R)}}{(2\im (\l^2))^{\a_2-\a_1}}, \quad \forall\, u\in H^{\a_2}(\R).
		\end{equation}	
		\item If $2\le p< \infty$, then
		\begin{equation}
		\left\| (D\pm 2\l^2)^{-1} Du \right\|_{L^p(\R)} \lesssim_p \|u\|_{H^{s^*(p)}(\R)}, \quad \forall\, u\in H^{s^*(p)}(\R).	
		\end{equation}
	\end{itemize}
\end{LEMMA}

We estimate one of the terms defining~$\t_0^2(u,\l)$ explicitly; the others can be dealt with in an entirely similar way.  
\begin{align*}
& \left| \frac{1}{\l^2} \int_\R \overline{u}(x) ((D + 2\l^2)^{-1}Du(x))^2 ((D - 2\l^2)^{-1}D\overline{u}(x))\dx\right| \\
& \quad \le \frac{1}{|\l|^2} \|u\|_{L^6(\R)} \|(D + 2\l^2)^{-1}Du\|_{L^6(\R)}^2 \|(D - 2\l^2)^{-1}D\overline{u}\|_{L^2(\R)} 
\lesssim_\a \frac{\|u\|_{H^{\frac13}(\R)}^3 \|u\|_{H^\a(\R)}}{|\l|^{2+2\a}}.
\end{align*}

We conclude that~$\t_0^2(u,\l)$ satisfies the required bound, finishing the case~$L=0$.

\subsection{Expanding the Resolvent}

\label{ss:ExpandRes}

\subsubsection{Formal Expansion of $R^a$ and $R^d$}

\label{sss:RaRdformalexp}

As stated above, for~$L\ge 1$ we seek an expansion of the symbol of~$L_u^{-1}(\l)$, in powers of~$\l^{-2}$.  That is, we seek to understand~$R(x,\zeta)$ in the expression
\begin{equation}
L_u^{-1}(\l)f = \frac{1}{\sqrt{2\pi}} \int \dd\zeta e^{ix\zeta} R(x,\zeta) \widehat{f}(\zeta).
\end{equation}	
The identity~$L_u(\l)R(x,D) = I$ implies
\begin{equation}
\label{e:Linverse}
i \s_3 \p_x R(x,\zeta) - (\zeta \s_3 + \l^2)R(x,\zeta) - i\l U(x)R(x,\zeta) = I.
\end{equation}	
Introducing the new variable~$p = \frac{\zeta}{\l^2}$, this reads 	
\begin{equation}
\label{e:Linversep}
i \s_3 \p_x R(x,\zeta) - \l^2(p \s_3 + 1)R(x,\zeta) - i\l U(x)R(x,\zeta) = I.
\end{equation}		

We split~$R$ into its diagonal and antidiagonal parts~$R^d$ and~$R^a$, respectively,
\[
R(x,\zeta) = R^d(x,\zeta) + R^a(x,\zeta),
\]
and we also split equation~\eqref{e:Linversep} accordingly: 
\begin{equation}
\label{e:diag}
i \s_3 \p_x R^d(x,\zeta) - \l^2 (p \s_3 + 1)R^d(x,\zeta) - i\l U(x)R^a(x,\zeta) = I;
\end{equation}
\begin{equation}
\label{e:antidiag}
i \s_3 \p_x R^a(x,\zeta) - \l^2 (p \s_3 + 1)R^a(x,\zeta) - i\l U(x)R^d(x,\zeta) = 0.
\end{equation}

Setting the notation
\[
R^d(x,\zeta) = \sum_{k\ge 0} \frac{1}{\l^{2+2k}} R^d_k(x,p), 
\qquad 
R^a(x,\zeta) = \sum_{k\ge 0} \frac{1}{\l^{3 + 2k}} R^a_k(x,p),
\]
we rewrite~\eqref{e:diag} and~\eqref{e:antidiag} in expanded form:
\begin{align}
\label{e:diagexp}
I & = -(p\s_3 + 1) R_0^d + \sum_{k=1}^\infty \frac{i\s_3 \p_x R_{k-1}^d - (p \s_3 + 1)R_{k}^d - iUR_{k-1}^a}{\l^{2k}};\\
\label{e:antidiagexp}
0 & = -(p\s_3 + 1)R_0^a - iU R_0^d + \sum_{k=1}^\infty \frac{ i\s_3 \p_x R_{k-1}^a - (p\s_3 + 1)R_{k}^a - iU R_{k}^d}{\l^{2k}}.
\end{align}
We thus obtain the recursive system~\eqref{e:baserecursive}--\eqref{e:newrecursivea} below.
\begin{equation}
\label{e:baserecursive}
R_0^d(x,p) = -\frac{p\s_3-1}{p^2-1},
\qquad 
R_0^a(x,p) = -\frac{iU}{p^2-1}, 
\end{equation}
\begin{equation}
\label{e:newrecursived}
R_k^d(x,p) = \frac{1}{p^2-1} \big[-iU R_{k-1}^a(x,p) + i\p_x R_{k-1}^d(x,p)\s_3 \big] (p\s_3-1), \qquad\qquad \qquad \qquad k\ge 1,
\end{equation}
\begin{equation}
\label{e:newrecursivea}
\begin{split}
R_k^a(x,p) & = \frac{1}{p^2 - 1} \big[ iU R_k^d(x,p) + i\p_x R_{k-1}^a(x,p) \s_3 \big](p\s_3 + 1) \\
& = \frac{1}{p^2-1} \big[ U^2 R_{k-1}^a(x,p) - U\p_x R_{k-1}^d(x,p)\s_3 + i\p_x R_{k-1}^a(x,p)\s_3(p\s_3+1) \big],
\end{split}
\quad k\ge 1.
\end{equation}
Note that we used the formula for~$R_k^d(x,p)$ to pass to the second line in the formula for~$R_k^a(x,p)$.  We also used several times the fact that~$\s_3 A = -A\s_3$ for any antidiagonal matrix.

We use the computations above to clarify the form of the~$R_k^d$'s and~$R_k^a$'s; the precise statement is contained in the following Lemma. 

\begin{LEMMA}
	\label{l:Rkform}
	The~$R_k^d$'s and~$R_k^a$'s take the following form:
	\begin{equation}
	R_k^d(x,p) = \sum_{r=1}^k R_{k,r}^d(x,p), \qquad k\ge 1,
	\end{equation}
	\begin{equation}
	R_k^a(x,p) = \sum_{r=0}^k R_{k,r}^a(x,p), \qquad k\ge 0,
	\end{equation}	
	where the entries of the~$R_{k,r}^d$'s and~$R_{k,r}^a$'s are homogeneous polynomials of degrees~$2r$ and~$2r+1$, respectively, in~$u, \overline{u}$, and their derivatives.  More specifically, setting $Q_{\g} = \p_x^{\g_1} U \cdots \p_x^{\g_n} U$, for~$\g\in \N^n$, we have
		\begin{align}
		\label{e:Rdkr}
		R^d_{k,r}(x,p) 
		& = \frac{1}{(p^2-1)^{k+1}} 
		\sum_{\substack{\g \in \N^{2r} \\|\g| = k-r}} 
		Q_{\g}(x) P_{|\g|}(p)(p\s_3-1),\\
		\label{e:Rakr}
		R^a_{k,r}(x,p)  
		& = \frac{1}{(p^2-1)^{k+1}} 
		\sum_{\substack{\g\in \N^{2r+1} \\|\g| = k-r}}
		Q_{\g}(x) P_{|\g|}(p).
		\end{align}
		Here and below we use the notation $P_n$  to denote any diagonal matrix whose diagonal entries are polynomials in~$p$ having degree at most $n$.
\end{LEMMA}

We postpone the proof of this Lemma until Section~\ref{ss:induction}, so as not to interrupt the flow of ideas.

\subsubsection{The Truncated Expansion, and a Formula for $\cR_{2m-1}$}
\label{sss:truncated}

For a fixed~$N\in \N^*$, we set the following notation.  (Later we will set~$N = 2L$.)
\begin{equation}
\label{e:truncatedexp}
\begin{split}
R^{(N)}(x,p) & = \underbrace{\sum_{k=0}^N \frac{R_k^d(x,p)}{\l^{2+2k}}}_{=:R_d^{(N)}(x,p)} + \underbrace{\sum_{k=0}^{N-1} \frac{R_k^a(x,p)}{\l^{3+2k}} }_{=:R_a^{(N)}(x,p)} \\
& = \underbrace{\frac{R_0^d(x,p)}{\l^2}}_{=:R^{(N)}_{d,0}(x,p)} + \sum_{r=1}^N \underbrace{\sum_{k=r}^N \frac{R_{k,r}^d(x,p)}{\l^{2+2k}}}_{=:R^{(N)}_{d,r}(x,p)} + \sum_{r=0}^{N-1} \underbrace{\sum_{k=r}^{N-1} \frac{ R_{k,r}^a(x,p)}{\l^{3+2k}}}_{=:R^{(N)}_{a,r}(x,p)}.
\end{split}
\end{equation}

The symbol~$R^{(N)}(x,p)$ is a truncated expansion of~$R(x,p)$ in inverse powers of~$\l$, having diagonal and antidiagonal parts~$R^{(N)}_d$,~$R^{(N)}_a$, respectively.  The point of this definition is that, using Lemma~\ref{l:Rkform}, we know that~$R^{(N)}_{d,r}$ is homogeneous of degree~$2r$ in~$u$,~$\overline{u}$, and their derivatives, while~$R^{(N)}_{a,r}$ is homogeneous of degree~$2r+1$ in these quantities.  Expanding $R^{(N)}$  according to \eqref{e:truncatedexp} and applying the recursive identities~\eqref{e:diagexp}--\eqref{e:antidiagexp}, we see that~$R^{(N)}(x,p)$ satisfies
\begin{equation*}
[i \s_3 \p_x  - \l^2(p \s_3 + 1) - i\l U(x)]R^{(N)}(x,p) = I+ Y^{(N)}(x,p),
\end{equation*}	
where
$Y^{(N)}(x,p)=Y^{(N)}_d(x,p)+Y^{(N)}_a(x,p)$, 
\[
Y^{(N)}_d(x,p)= \frac{1}{\l^{2+2N}} i\s_3  (\p_x R_N^d)(x,p), \quad Y^{(N)}_a(x,p)= -\frac{1}{\l^{1+2N}} R_N^a(x,p)(p\s_3 - 1).
\]
This implies
\begin{equation}
\label{e:RYN}
L_u^{-1}(\l) = R^{(N)}(x,\l^{-2} D) - L_u^{-1}(\l) Y^{(N)}(x,\l^{-2} D).
\end{equation}

Recall that~$\cR_{2m-1}$ is the term in the expansion~\eqref{e:Luinvexp} which is homogeneous of order~$2m-1$ in~$u$,~$\overline{u}$, and their derivatives.  On the other hand, the portion of~$R^{(N)}$ which is of this homogeneity is precisely~$R^{(N)}_{a,m-1}$.  Combining these considerations with~\eqref{e:RYN}, we see that~$\cR_{2m-1}$ is the difference between~$R^{(N)}_{a,m-1}$ and the part of~$L_u^{-1}(\l)Y^{(N)}(x,\l^{-2}D)$ that is homogeneous of degree~$2m-1$ in~$u$,~$\overline{u}$, and their derivatives.  
Using the expansion \eqref{e:Luinvexp} to isolate this part, we obtain:

\begin{equation}
\label{e:R2m-1}
\begin{split}
\cR_{2m-1}& = R_{a,m-1}^{(N)}(x,\l^{-2}D) - \frac{1}{\l^{2+2N}}\sum_{\substack{k+r' = m-1\\k\ge 0,\;1\le r'\le N}}  T_u(\l)^{2k+1} (\cL_0 - \l^2)^{-1} (\cL_0  R_{N,r'}^{d})(x,\l^{-2}D) \\ 
& \quad - \frac{1}{\l^{1+2N}} \sum_{\substack{k+r'=m-1\\ k\ge 0,\;0\le r'\le N}} T_u(\l)^{2k} (\cL_0 - \l^2)^{-1} R_{N,r'}^a(x,\l^{-2}D) (\l^{-2}\cL_0 + 1).
\end{split}
\end{equation}

\subsection{Extracting the $\mu_{j,m}(u)$'s}
\label{ss:extractmuj}

Combining~\eqref{e:R2m-1} with~\eqref{e:TuversusR},~\eqref{e:truncatedexp}, and pulling out inverse powers of~$\l$, we easily find the following formula for~$\tr T_u^{2m}(\l)$  with $m\ge 2$, truncated at~$N = 2L$.
\begin{small}
	\begin{equation}
	\label{e:decompstep2}
	\begin{split}
	\tr(T_u^{2m}(\l)) = & \; \sum_{j=m-1}^{2L-1} \frac{1}{\l^{2j+2}} \tr[iU  R_{j,m-1}^a(x,\l^{-2} D)]\\ 
	& + \sum_{\substack{k+r = m-1\\ k\ge 0,\;1 \le r \le 2L}} \frac{(-1)^k}{\l^{4L + 4 + 2k}}  \tr\big[(U(\l^{-2}\cL_0-1)^{-1})^{2k+2}  (\cL_0 R_{2L,r}^d)(x,\l^{-2}D)\big] \\
	& + \sum_{\substack{k+r = m-1 \\ k\ge 0, \;0\le r \le 2L}} \frac{i(-1)^{k+1}}{\l^{4L + 2+ 2k}} \tr\big[ (U(\l^{-2}\cL_0 - 1)^{-1})^{2k+1}  R_{2L,r}^a(x,\l^{-2}D)(\l^{-2}\cL_0 + 1)\big].
	\end{split}
	\end{equation}
\end{small}
We will refer to the three sums above as~$I$,~$II$, and~$III$, respectively. 

We now  identify the coefficients~$\mu_{j,m}(u)$'s and verify that they satisfy the properties claimed in Lemma~\ref{l:trdecomp}.  The claimed homogeneity properties will be clear from the formulas that we derive below; we will just need to verify the bounds ~\eqref{e:mu2lbd}.  The latter are also straightforward to verify but will require us to use the structure of the~$R^d_{k,r}$'s and~$R^a_{k,r}$'s from~\eqref{e:Rdkr}--\eqref{e:Rakr}.  

The first sum has the form
$$-2m\sum_{j=m-1}^{2L-1} \frac{\mu_{j,m}(u)}{\l^{2j}} $$
with
\begin{equation}
\label{e:mujkall}
\begin{split}
\mu_{j,m}(u)
& = -\frac{i}{2m\l^2} \tr \big[ UR_{j,m-1}^a(x,\l^{-2}D) \big]  \\
& = -\frac{i}{4m\pi} \sum_{\substack{\g\in \N^{2m-1} \\ |\g| = j - (m-1)}} \tr \left[ \int U(x) Q_\g(x)\dx \int \frac{P_{|\g|}(\tfrac{\zeta}{\l^2})}{((\tfrac{\zeta}{\l^2})^2 - 1)^{j+1}} \frac{\dd\zeta}{\lambda^2} \right] .
\end{split}
\end{equation}

Note that~`$\tr$' denotes an operator trace in the first line, whereas it refers to the~$2\times 2$ matrix trace in the second and third lines.  We will use the notation~`$\tr$' similarly in what follows without further comment.  
  
Since~$\l$ is presumed to lie in~$\Gamma_\d$, a comparison of the degrees in the numerator and denominator ensures that the integrals over~$\zeta$ are finite and their values are independent of $\lambda$.  The total number of derivatives in the~$x$-integrals is~$j-(m-1)$, and we distribute these so that the highest order of the derivatives that fall on a single~$U$ is as small as possible. We list the bounds on~$\mu_{j,m}(u)$ according to~$m$ and the parity of~$j$.  In each case below,~$\ell$ is a strictly positive integer.
\begin{itemize}
	\item If~$j=2\ell$ is even and~$m=2$, then there are~$2\ell-1$ derivatives; thus 
	\[
	|\mu_{2\ell,2}(u)|\lesssim \|u\|_{H^\ell(\R)} \|u\|_{H^{\s(\ell-1)}(\R)}^3,
	\]
	where we recall the notation $\s(n) = \max\{n, \frac13\}$.
	\item If~$j = 2\ell$ is even and~$m\ge 3$, then there are at most~$2(\ell-1)$ total derivatives.  This establishes the following bounds: 
	\[
	\begin{split}
	|\mu_{2,3}(u)| & \lesssim \|u\|_{H^{\frac13}(\R)}^{6},\\
	|\mu_{2\ell,m}(u)| & \lesssim \|u\|_{H^{\ell-1}(\R)}^{2m}, \qquad \ell\ge 2, \;m\in \{3, \ldots, 2\ell + 1\}.
	\end{split}
	\]
	\item If~$j = 2\ell+1$ is odd and~$m\ge 2$, then there are at most~$2 \ell$ derivatives, so~$|\mu_{2\ell+1,m}(u)|\lesssim \|u\|_{H^\ell(\R)}^{2m}$. 
\end{itemize}

Let us remark that the  formula~\eqref{e:mujkall} determines $\mu_{j,m}(u)$ for all $m\ge 2$ and all $j\ge m-1$, not just for those $\mu_{j,m}(u)$'s that appear in the sum $I$.  Therefore, the above bounds on the $\mu_{j,m}(u)$'s complete our proof of the estimates \eqref{e:mu2lbd}.  However, in order to determine the remainders~$\t_{L}^m(u,\l)$, we still need to extract~$\mu_{2L,m}(u)$ and $\mu_{2L+1,m}(u)$ from the sums~$II$ and~$III$.  To this end, we remove the parts of~$II$ and~$III$ which are of order $\l^{-4L}$ and $\l^{-4L-2}$; these will correspond to~$\mu_{2L,m}(u)$ and~$\mu_{2L+1,m}(u)$, respectively.  We deal first with~$\mu_{2L,m}(u)$; the only term expected to be relevant is the~$k=0$ term in~$III$, namely 
\begin{equation}
\label{e:IIIk=0}
-\frac{i}{\l^{4L+2}}\tr\big[ U(\l^{-2}\cL_0 - 1)^{-1} R_{2L,m-1}^a(x,\l^{-2}D)(\l^{-2}\cL_0 + 1)\big].
\end{equation}
To extract the part of this expression that is really of order~$\l^{-4L}$, we commute the operator~$(\l^{-2}\cL_0 - 1)^{-1}$ with~$R_{2L,m-1}^a(x,\l^{-2}D)$.
We will have to do something similar several times below, so let us pause to write down a more general formula.  Let~$A$ denote an antidiagonal operator with symbol~$A(x,\zeta)$ and similarly let~$B$ denote a diagonal operator with symbol~$B(x,\zeta)$.  Then a simple application of the product rule gives the following operator identities.

 \begin{align}
\label{e:Acommute}
A(\cL_0 +\l^2)^{-1} & = -(\cL_0 -\l^2)^{-1}A + (\cL_0 -\l^2)^{-1}(\cL_0 A)(\cL_0 +\l^2)^{-1}; \\
\label{e:Dcommute}
B(\cL_0 -\l^2)^{-1} & = \;\;\,(\cL_0 -\l^2)^{-1}B + (\cL_0 -\l^2)^{-1}(\cL_0 B)(\cL_0 -\l^2)^{-1}. 
\end{align}
Using~\eqref{e:Acommute} with~$A = R_{2L,m-1}^a$, the expression~\eqref{e:IIIk=0} becomes 
\begin{equation}
\label{e:IIIk=0'}
\begin{split}
\underbrace{\frac{i}{\l^{4L+2}} \tr\big[ U R_{2L,m-1}^a(x,\l^{-2}D) \big]}_{-2m\mu_{2L, m}\l^{-4L}} - \frac{i}{\l^{4L+4}}\tr\big[ U(\l^{-2}\cL_0 -1)^{-1} (\cL_0R_{2L,m-1}^a)(x,\l^{-2}D)\big].
\end{split}
\end{equation}

\medskip
To extract~$\mu_{2L+1,m}(u)$, we need to determine the part of~$II$ and~$III$ that is of order~$2L+1$ in~$\l^{-2}$.  There are three quantities we need to consider:
\begin{itemize}
	\item The second term in~\eqref{e:IIIk=0'}:
	\begin{equation}
	\label{e:IIIk=0''}
	-\frac{i}{\l^{4L+4}}\tr\big[  U(\l^{-2}\cL_0 -1)^{-1}( \cL_0R_{2L,m-1}^a)(x,\l^{-2}D)\big]
	\end{equation}
	\item The~$k=0$ term in~$II$:
	\begin{equation}
	\label{e:IIk=0}
	\frac{1}{\l^{4L + 4}}  \tr\big[(U(\l^{-2}\cL_0-1)^{-1})^{2}  (\cL_0 R_{2L,m-1}^d)(x,\l^{-2}D)\big]
	\end{equation}
	\item The~$k=1$ term in~$III$:
	\begin{equation}
	\label{e:IIIk=1}
	\frac{i}{\l^{4L + 4}} \tr\big[ (U(\l^{-2}\cL_0 - 1)^{-1})^{3}  R_{2L,m-2}^a(x,\l^{-2}D)(\l^{-2}\cL_0 + 1)\big].
	\end{equation}
\end{itemize}
We deal with each of these in turn, denoting their contributions to~$\mu_{2L+1,m}$ by~$\mu_{2L+1,m}^{(1)}$,~$\mu_{2L+1,m}^{(2)}$, and~$\mu_{2L+1,m}^{(3)}$, respectively.    To put~\eqref{e:IIIk=0''} in the desired form, we simply apply~\eqref{e:Acommute} again, this time with~$A = \cL_0R_{2L,m-1}^a$.  The result is
\begin{equation}
\label{e:2L+1(i)}
\begin{split}
& \frac{i}{\l^{4L+4}}\tr\big[ U(\cL_0 R_{2L,m-1}^a)(x,\l^{-2}D) (\l^{-2}\cL_0 + 1)^{-1} \big] \\
& - \frac{i}{\l^{4L+6}}\tr\big[ U(\l^{-2}\cL_0 -1)^{-1} (D^2R_{2L,m-1}^a)(x,\l^{-2}D) (\l^{-2}\cL_0 + 1)^{-1}\big].
\end{split}
\end{equation}
Thus 
\begin{align*}
\mu_{2L+1,m}^{(1)}  = -\frac{i}{2m\l^2}\tr\big[ U(\cL_0 R_{2L,m-1}^a)(x,\l^{-2}D) (\l^{-2}\cL_0 + 1)^{-1} \big] .
\end{align*}

\medskip
Next, we look at~\eqref{e:IIk=0}.  We perform two commutations, using~$A = U$ in~\eqref{e:Acommute}, then~$B = U^2$ in~\eqref{e:Dcommute} to obtain
\begin{equation}
\label{e:commutetwice}
\begin{split}
[U(\l^{-2}\cL_0 - 1)^{-1}]^2 
& = -(\l^{-4} D^2 - 1)^{-1} U^2 \\
& \qquad + \l^{-2}(\l^{-2} \cL_0 + 1)^{-1} (\cL_0 U)(\l^{-2} \cL_0 - 1)^{-1} U (\l^{-2} \cL_0 - 1)^{-1} \\
& \qquad - \l^{-2} (\l^{-4}D^2 - 1)^{-1} (\cL_0 U^2)(\l^{-2}\cL_0 - 1)^{-1}.
\end{split}
\end{equation}
Substituting this into~\eqref{e:IIk=0} yields 
\begin{equation}
\label{e:2L+1(ii)}
\begin{split}
& -\frac{1}{\l^{4L + 4}}  \tr\big[ U^2 (\cL_0 R_{2L,m-1}^d)(x,\l^{-2}D)(\l^{-4} D^2 - 1)^{-1}\big] \\
& +\frac{1}{\l^{4L + 6}}  \tr\big[(\l^{-2} \cL_0 + 1)^{-1} (\cL_0 U)(\l^{-2} \cL_0 - 1)^{-1} U (\l^{-2} \cL_0 - 1)^{-1} (\cL_0 R_{2L,m-1}^d)(x,\l^{-2}D)\big] \\
& -\frac{1}{\l^{4L + 6}}  \tr\big[(\l^{-4}D^2 - 1)^{-1}(\cL_0 U^2)(\l^{-2}\cL_0 - 1)^{-1}  (\cL_0 R_{2L,m-1}^d)(x,\l^{-2}D)\big].
\end{split}
\end{equation}
We take 
\begin{align*}
\mu_{2L+1,m}^{(2)}
= \frac{1}{2m\l^2} \tr\big[ U^2 (\cL_0 R_{2L,m-1}^d)(x,\l^{-2}D)(\l^{-4} D^2 - 1)^{-1}\big].
\end{align*}

\medskip
Finally, we look at~\eqref{e:IIIk=1}.  
Proceeding as in~\eqref{e:commutetwice} but commuting one more time, we get
\begin{align*}
(\l^{-2}\cL_0 +1) [U(\l^{-2}\cL_0 - 1)^{-1}]^3 
& = (\l^{-4}D^2 - 1)^{-1} U^3 \\
& \qquad + \l^{-2} (\cL_0 U) (\l^{-2} \cL_0 - 1)^{-1} U (\l^{-2} \cL_0 - 1)^{-1} U (\l^{-2}\cL_0 - 1)^{-1}\\
& \qquad - \l^{-2} (\l^{-2}\cL_0 - 1)^{-1}  (\cL_0 U^2) (\l^{-2} \cL_0 - 1)^{-1} U(\l^{-2}\cL_0 - 1)^{-1} \\
& \qquad - \l^{-2} (\l^{-4}D^2 - 1)^{-1} (\cL_0 U^3) (\l^{-2} \cL_0 - 1)^{-1}.
\end{align*}

Substituting the above into~\eqref{e:IIIk=1} yields 
\begin{equation}
\label{e:2L+1(iii)}
\begin{split}
& \frac{i}{\l^{4L + 4}} \tr\big[  U^3  R_{2L,m-2}^a(x,\l^{-2}D)(\l^{-4} D^2 - 1)^{-1} \big] \\
& +\frac{i}{\l^{4L + 6}} \tr\big[ (\cL_0 U) (\l^{-2} \cL_0 - 1)^{-1} U (\l^{-2} \cL_0 - 1)^{-1}U (\l^{-2} \cL_0 - 1)^{-1} R_{2L,m-2}^a(x,\l^{-2}D)\big] \\
& -\frac{i}{\l^{4L + 6}} \tr\big[ (\l^{-2}\cL_0 - 1)^{-1}  (\cL_0 U^2) (\l^{-2} \cL_0 - 1)^{-1} U(\l^{-2}\cL_0 - 1)^{-1} R_{2L,m-2}^a(x,\l^{-2}D) \big] \\
& -\frac{i}{\l^{4L + 6}} \tr\big[  (\l^{-4} D^2 - 1)^{-1} (\cL_0 U^3) (\l^{-2} \cL_0 - 1)^{-1}  R_{2L,m-2}^a(x,\l^{-2}D) \big].
\end{split}
\end{equation}
Thus

\begin{align*}
\mu_{2L+1,m}^{(3)} = -\frac{i}{2m\l^2} \tr\big[ U^3  R_{2L,m-2}^a(x,\l^{-2}D)(\l^{-4} D^2 - 1)^{-1} \big].
\end{align*}

A short calculation involving \eqref{e:newrecursivea} confirms that the three quantities identified above sum to $\mu_{2L+1,m}$ as defined in \eqref{e:mujkall}:
$$\mu_{2L+1,m}^{(1)}+\mu_{2L+1,m}^{(2)}+\mu_{2L+1,m}^{(3)}=\mu_{2L+1,m}.$$

\subsection{Estimating the Remainder}

\label{ss:estrem}

The final step of the proof is to estimate the remainder terms, which we group together into~$\t^{m}_{L}(u,\l)$.  This expression is a sum of the following terms:
\begin{itemize}
	\item The~$k\ge 1$ terms of~$II$ and the~$k\ge 2$ terms of~$III$, where~$II$ and~$III$ denote (as above) the second and third sums in the decomposition~\eqref{e:decompstep2}.  We refer to these as the `Type 1' remainder terms.  
	\item The terms in~\eqref{e:2L+1(i)},~\eqref{e:2L+1(ii)}, and~\eqref{e:2L+1(iii)} where~$\l^{-4L-6}$ appears (six terms total).  We refer to these as the `Type 2' remainder terms.  
\end{itemize}

\subsubsection{Type 1 Remainder Terms}
We begin with the two sums.  We want to show that the following expression is bounded by~$\|u\|_{H^L(\R)}^{2m}$:
\begin{equation}
\label{e:remsums}
\begin{split}
& \sum_{\substack{k+r = m-1\\ k\ge 1,\;1 \le r \le 2L}} \frac{(-1)^k}{\l^{2k}}  \tr\big[(U(\l^{-2}\cL_0-1)^{-1})^{2k+2}  (\cL_0 R_{2L,r}^d)(x,\l^{-2}D)\big] \\
& + \sum_{\substack{k+r = m-1 \\ k\ge 2, \;0\le r \le 2L}} \frac{i(-1)^{k+1}}{\l^{2(k-1)}} \tr\big[ (U(\l^{-2}\cL_0 - 1)^{-1})^{2k+1} R_{2L,r}^a(x,\l^{-2}D)(\l^{-2}\cL_0 + 1)\big].
\end{split}
\end{equation}
By virtue of \eqref{e:Rdkr}, we can write
\begin{align*}
(\cL_0 R^d_{2L,r})(x,p)  = \frac{1}{(p^2 - 1)^{2L+1}}\sum_{\substack{\g\in \N^{2r} \\ |\g| = 2L-r+1}}Q_\g(x) P_{|\g|}(p).
\end{align*}
Thus, the traces in the first sum in~\eqref{e:remsums} may be written as a sum of terms of the form
\begin{align}\label{v}
 \tr\big[(U(\l^{-2}\cL_0-1)^{-1})^{2k+2} Q_\g(x) P_{2L-r+1}(\l^{-2}D)(\l^{-4}D^2 - 1)^{-2L-1} \big]
\end{align}
with $\g\in \N^{2r},\,\,|\g| = 2L-r+1\leq 2L$.
Integrating by parts repeatedly in the above expression until no derivative of order larger than~$L$ falls on any single~$U$, we rewrite the expression \eqref{v} as a sum of terms of the form
\[
\tr\big[(\partial_x^{\eta_1}U)(x)(\l^{-2}\cL_0-1)^{-1}\dots(\partial_x^{\eta_{2k+2}}U)(x)(\l^{-2}\cL_0-1)^{-1}Q_\g(x) P_{2L-r+1}(\l^{-2}D)(\l^{-4}D^2 - 1)^{-2L-1} \big],
\]
with $\eta=(\eta_1, \dots\eta_{2k+2})\in \N^{2k+2}$, $\g=(\g_1, \dots , \g_{2r})\in \N^{2r}$ satisfying $|\eta|+|\g|=2L-r+1$ and~$\max\limits_{p,q}(\eta_p, \g_q)\leq L$. Thus, the expression \eqref{v} can be bounded by $|\l|^2\|u\|_{H^L(\R)}^{2m}$, and therefore the first sum in \eqref{e:remsums} by $\|u\|_{H^L(\R)}^{2m}$.
 
We deal with the second sum in~\eqref{e:remsums} in essentially the same way.  Invoking \eqref{e:Rakr}, we may write 
\[
R^a_{2L,r}(x,\l^{-2}D)(\l^{-2}\cL_0 + 1)=
 \sum_{\substack{\g\in \N^{2r+1} \\ |\g| = 2L-r}}Q_{\g}(x)P_{|\g|+1}(\l^{-2}D) (\l^{-4}D^2-1)^{-2L-1}.
\]
Thus
\[
\tr\big[ (U(\l^{-2}\cL_0 - 1)^{-1})^{2k+1} R_{2L,r}^a(x,\l^{-2}D)(\l^{-2}\cL_0 + 1)\big]
\]
is a sum of terms of the form 
\begin{equation}
\label{e:sum2preIBP}
\tr\big[ (U(\l^{-2}\cL_0 - 1)^{-1})^{2k+1} Q_{\g}(x)P_{2L-r+1}(\l^{-2}D) (\l^{-4}D^2-1)^{-2L-1} \big],
\end{equation}
where~$\g\in \N^{2r+1}$ with~$|\g|=2L-r\leq 2L$. As before, we integrate by parts repeatedly to rewrite the expression \eqref{e:sum2preIBP} as a sum of terms of the form
$$\tr\big[(\partial_x^{\eta_1}U)(x)(\l^{-2}\cL_0-1)^{-1}\dots(\partial_x^{\eta_{2k+1}}U)(x)(\l^{-2}\cL_0-1)^{-1}Q_\g(x) P_{2L-r+1}(\l^{-2}D)(\l^{-4}D^2 - 1)^{-2L-1} \big],$$
with $\eta=(\eta_1, \dots\eta_{2k+1})\in \N^{2k+1}$, $\g=(\g_1, \dots , \g_{2r+1})\in \N^{2r+1}$ satisfying
$|\eta|+|\g|=2L-r$ and~$\max\limits_{p,q}(\eta_p, \g_q)\leq L$. This allows us to bound \eqref{e:sum2preIBP} by $|\l|^2\|u\|_{H^L(\R)}^{2m}$, thus completing the desired estimates on the Type 1 remainder terms.

\subsubsection{Type 2 Remainder Terms}

We now deal with the Type~2 remainder terms (the terms in~\eqref{e:2L+1(i)},~\eqref{e:2L+1(ii)}, and~\eqref{e:2L+1(iii)} where~$\l^{-4L-6}$ appears).  When~$m\ge 3$, the total number of derivatives falling on the~$U$'s is~$2L$; therefore we can bound all these terms by~$|\l|^{-4L -4}\|u\|_{H^L(\R)}^{2m}$ by arguing exactly as we did for the Type~1 terms.  To complete the proof of  Lemma~\ref{l:trdecomp}, it thus remains to consider the Type~2 remainder terms with~$m = 2$. In this case, some of the~$U$'s appear to be overloaded with derivatives, and we need an additional estimate.  We state the following Lemma in terms of the `overloaded' part of the Type~2 remainder term from~\eqref{e:2L+1(i)}, but the same manipulations will yield the bound we need for the other Type~2 remainders.   

\begin{LEMMA}
	\label{l:overload}
	The following estimate holds, for any~$j\in \N$ and all~$\a\in [0,1]$.
	\begin{equation}
	\|(\cL_0 +\l^2)^{-1} D^{j+1} U(\cL_0 - \l^2)^{-1} \|_2 \le C(\a) |\l|^{-1-2\a} \|u\|_{H^{j + \a}(\R)}. 
	\end{equation}
\end{LEMMA}
\begin{proof}
Denoting~$T = (\cL_0 +\l^2)^{-1}D^{j+1}U(\cL_0 - \l^2)^{-1}$, we readily compute as follows:
\begin{small}
	\begin{align*}
	\|T\|_2^2
&= \frac{1}{\pi}  \iint \frac{|\widehat{D^{j+1}u}(\zeta_1 - \zeta_2)|^2 }
	{|\zeta_1 - \l^2|^{2} |\zeta_2 - \l^2|^{2}}\dd\zeta_1\dd\zeta_2\\
	& = \frac{1}{\pi} \int |\zeta_1|^2 |\widehat{D^j u}(\zeta_1)|^2 \left(  \int \frac{\dd\zeta_2 }{|\zeta_1 + \zeta_2 - \l^2|^2 |\zeta_2 - \l^2|^2} \right) \dd\zeta_1 \\
	& = \frac{2}{\im \l^2} \int \frac{|\zeta_1|^2 |\widehat{D^j u}(\zeta_1)|^2}{|\zeta_1 + 2i \im \l^2|^2} \dd\zeta_1 
	 \le C(\a) |\l|^{-2-4\a} \|u\|_{H^{j+\a}(\R)}^2.
	\end{align*}
	\end{small}
	This completes the proof.	
\end{proof}

With the above Lemma at our disposal, we now return to the estimation of the Type~$2$ remainder terms for $m=2$; we deal first with the one that appears in~\eqref{e:2L+1(i)}. Omitting the prefactor $\frac{-i}{\l^{4L+6}}$, the quantity under consideration is
$$
\tr\big[ U(\l^{-2}\cL_0 -1)^{-1} (D^2R_{2L,1}^a)(x,\l^{-2}D) (\l^{-2}\cL_0 + 1)^{-1}\big],
$$ 
which we write as
\begin{equation}\label{i}
\tr\big[ (\l^{-2}\cL_0 +1)^{-1}(D^2 U)(\l^{-2}\cL_0 -1)^{-1} R_{2L,1}^a(x,\l^{-2}D)\big].\end{equation}
Proceeding as we did for the Type 1 terms, 
we rewrite this expression as a sum of terms of the form
\begin{equation}
\label{e:type2term1}
\tr\big[ (\l^{-2}\cL_0 +1)^{-1}(D^{\eta+2} U)(\l^{-2}\cL_0 -1)^{-1} Q_\g(x) P_{2L-1}(\l^{-2}D)(\l^{-4}D^2 - 1)^{-2L-1}],
\end{equation}
with $\eta\in \N$, $\g=(\g_1, \g_2, \g_3)\in \N^{3}$, $\eta+|\g|=2L-1$, $\eta\leq L-1$, and $\max(\g_1, \g_2, \g_3)\leq L$. The expression \eqref{e:type2term1} can be bounded by
\begin{align*}
\| (\l^{-2}\cL_0 +1)^{-1}(D^{\eta + 2} U)(\l^{-2}\cL_0 -1)^{-1} \|_2 \|Q_\g(x) P_{2L-1}(\l^{-2}D)(\l^{-4}D^2 - 1)^{-2L-1}\|_2.
\end{align*}
By virtue of Lemma \ref{l:overload}, the above can bounded in turn by 
\begin{align*}
C(\a) |\l|^{4-2\alpha}\|u\|_{H^{L+\a}(\R)} \|u\|_{H^L(\R)}^{3},
\end{align*}
for any $\alpha\in [0, 1]$.

The next quantity we treat is the third term in~\eqref{e:2L+1(ii)}; we want to estimate 
\begin{equation*}
\tr\big[(\l^{-2}\cL_0 - 1)^{-1}(\cL_0 U^2)(\l^{-2}\cL_0 - 1)^{-1}  (\cL_0 R_{2L,1}^d)(x,\l^{-2}D)(\l^{-2}\cL_0 + 1)^{-1}\big].
\end{equation*}
After an integration by parts we are left with the expression
\[
-\tr\big[(\l^{-2}\cL_0 - 1)^{-1}(D^2 U^2)(\l^{-2}\cL_0 - 1)^{-1}  R_{2L,1}^d(x,\l^{-2}D)(\l^{-2}\cL_0 + 1)^{-1}\big].
\]
that can be
treated in exactly the same way as \eqref{i}.  We note only the modification to Lemma~\ref{l:overload} that we use, namely 
\[
\| (\cL_0 - \l^2)^{-1}(D^{L+1} U^2)(\cL_0 - \l^2)^{-1}\|_2 \lesssim_\a |\l|^{-1-2\a} \|U^2\|_{H^{L+\a}(\R)} \lesssim_\a |\l|^{-1-2\a} \|u\|_{H^{L+\a}(\R)} \|u\|_{H^L(\R)}.
\]

We next consider the second term in~\eqref{e:2L+1(ii)}:
\[
 \tr\big[(\l^{-2} \cL_0 + 1)^{-1} (\cL_0 U)(\l^{-2} \cL_0 - 1)^{-1} U (\l^{-2} \cL_0 - 1)^{-1} 
(\cL_0 R_{2L,1}^d)(x,\l^{-2}D)\big] ,\]
where as usual we have suppressed  the prefactor~$\l^{-4L-6}$.
We start by rewriting it as the sum
\begin{equation}
\label{vvv}\begin{split}
-&\tr\big[(\l^{-2} \cL_0 + 1)^{-1} (D^2U)(\l^{-2} \cL_0 - 1)^{-1} U (\l^{-2} \cL_0 - 1)^{-1} 
 R_{2L,1}^d(x,\l^{-2}D)\big]\\
&+\tr\big[(\l^{-2} \cL_0 + 1)^{-1} (\cL_0 U)(\l^{-2} \cL_0 - 1)^{-1}(\cL_0 U) (\l^{-2} \cL_0 - 1)^{-1} 
 R_{2L,1}^d(x,\l^{-2}D)\big].
 \end{split}
\end{equation}
For the first term here we proceed exactly as before:  substituting \eqref{e:Rdkr} and integrating by parts we  rewrite it as a sum of expressions  of the form
$$\tr\big[ (\l^{-2}\cL_0 +1)^{-1}(D^{\eta_1+2} U)(\l^{-2}\cL_0 -1)^{-1}  (D^{\eta_2}U) (\l^{-2} \cL_0 - 1)^{-1} Q_\g(x) P_{2L}(\l^{-2}D)(\l^{-4}D^2 - 1)^{-2L-1}],$$
with~$\eta=( \eta_1, \eta_2)\in \N^2$,~$\g=(\g_1, \g_2)\in \N^{2}$,~$|\eta|+|\g|=2L-1$,~$\max(\eta_1, \eta_2)\leq L-1$, and~$\max(\g_1, \g_2)\leq~L$. We estimate  the above by  
\begin{align*}
\| (\l^{-2}\cL_0 +1)^{-1}(D^{\eta_1 + 2} U)(\l^{-2}\cL_0 -1)^{-1} \|_2\|  (D^{\eta_2}U) (\l^{-2} \cL_0 - 1)^{-1}\| \|Q_\g(x) P_{2L}(\l^{-2}D)(\l^{-4}D^2 - 1)^{-2L-1}\|_2,
\end{align*}
which can in turn be bounded by 
\begin{align*}
C(\a) |\l|^{4-2\alpha}\|u\|_{H^{L+\a}(\R)} \|u\|_{H^L(\R)}^{3}.
\end{align*}

To treat the second term in \eqref{vvv} we distinguish the cases~$L=1$ and~$L\geq 2$. In the case of~$L=1$ we estimate this expression by
\begin{align*}
&\| (\l^{-2}\cL_0 +1)^{-1}(\cL_0 U)\|_2  \|(\l^{-2} \cL_0 - 1)^{-1}(\cL_0 U)\| \|(\l^{-2} \cL_0 - 1)^{-1} R_{2L,1}^d(x,\l^{-2}D)\|_2\\
&\lesssim|\l|^{2+\frac2p}\|u\|_{H^1(\R)}^3\|Du\|_{L^p(\R)}, \quad 2\leq p\leq \infty.
\end{align*}
Putting~$\s = \frac{\a}{2}\in [0, \frac12[$ and choosing~$p$ such that~$\s = \frac12 - \frac1p$, we get the bound
 $$|\l|^{3 - 2\s} \|u\|_{H^1(\R)}^3\|u\|_{H^{1+\s}(\R)} \le |\l|^{4-2\a} \|u\|_{H^1(\R)}^3\|u\|_{H^{1 + \a}(\R)}.$$

If~$L\geq2$, we can proceed as for the Type 1 remainder terms and bound the second term in \eqref{vvv} by~$|\l|^2\|u\|_{H^L(\R)}^{4}$.
This finishes our considerations of Type 2 remainder terms coming from~\eqref{e:2L+1(ii)}. 

The last group of Type~2 remainder terms comes from~\eqref{e:2L+1(iii)}.  Omitting the common prefactor, the quantity of interest is 
\begin{equation}
\begin{split}
& \tr\big[ (\cL_0 U) (\l^{-2} \cL_0 - 1)^{-1} U (\l^{-2} \cL_0 - 1)^{-1}U (\l^{-2} \cL_0 - 1)^{-1} R_{2L,0}^a(x,\l^{-2}D)\big] \\
& -\tr\big[ (\l^{-2}\cL_0 - 1)^{-1}  (\cL_0 U^2) (\l^{-2} \cL_0 - 1)^{-1} U(\l^{-2}\cL_0 - 1)^{-1}  R_{2L,0}^a(x,\l^{-2}D) \big] \\
& -\tr\big[  (\l^{-4} D^2 - 1)^{-1} (\cL_0 U^3) (\l^{-2} \cL_0 - 1)^{-1} R_{2L,0}^a(x,\l^{-2}D) \big],
\end{split}
\end{equation}
where $R_{2L,0}^a(x,p)=\partial_x^{2L}U(x)P_{2L}(p)(p^2-1)^{-2L-1}$. No new ideas are involved in the estimation of these terms; we simply integrate by parts $L-1$ times to keep $L+1$ derivatives on $U$ coming from $R_{2L,0}^a$ and then apply Lemma~\ref{l:overload}. We omit the remaining details for these terms. Having now established the required bounds on the remainder~$|\t_L^{m}(u,\l)|$, we have completed the proof of Lemma~\ref{l:trdecomp}, modulo the proof of Lemma~\ref{l:Rkform} below.

\subsection{Proof of Lemma \ref{l:Rkform} }

\label{ss:induction}

We argue by induction.  The base cases are easy to verify explicitly:
\[
R_0^a(x,p) = \frac{1}{p^2-1}\cdot U(-i) = R_{0,0}^a(x,p) 
\]
\[
R_1^d(x,p) = \frac{1}{(p^2-1)^2}\cdot U^2\cdot (-1)\cdot (p\s_3-1) = R_{1,1}^d(x,p).
\]

Using~\eqref{e:newrecursived}--\eqref{e:newrecursivea} along with our inductive hypothesis, we may write, for~$k\ge 1$,
\begin{align*}
R_k^d(x,p)
& = \frac{1}{(p^2-1)^{k+1}} \bigg[ \sum_{r=0}^{k-1} \sum_{\substack{\g \in \N^{2r+1} \\ |\g| = (k-1)-r}} -iU Q_{\g}(x)P_{|\g|}(p) \\
& \hspace{25 mm} + \sum_{r=1}^{k-1} \sum_{\substack{\g\in \N^{2r} \\ |\g| = (k-1)-r}} i\p_x Q_{\g}(x)P_{|\g|}(p)( p\s_3-1)\s_3 \bigg](p\s_3-1).
\end{align*}
We check that the inner sums (together with the common factors of~$(p^2-1)^{-k-1}$ and~$(p\s_3-1)$) can be absorbed into~$R_{k,r+1}^d(x,p)$ and~$R_{k,r}^d(x,p)$, respectively.
\begin{itemize}
	\item When~$\g\in \N^{2r+1}$ and $|\g|=(k-1)-r$, the term~$iUQ_{\g}(x) P_{|\g|}(p)$ can be absorbed into~$R_{k,r+1}^d(x,p)$: 
	\begin{itemize}
		\item First,~$UQ_{\g} = Q_{(0,\g)}$, with~$(0,\g)\in \N^{2(r+1)}$ and~$|(0,\g)| = |\g| = k-(r+1)$.
		\item Second,~$\deg P_{|\g|}(p)\le (k-1)-r = k-(r+1)$.
	\end{itemize}   
	\item When~$\g\in \N^{2r}$ and~$|\g| = (k-1)-r$, the term~$\p_x Q_{\g}(x) P_{|\g|}(p)$ can be absorbed into~$R_{k,r}^d(x,p)$: 
	\begin{itemize}
		\item First,~$\p_x Q_{\g}$ is a sum of~$Q_{\g'}$'s, with~$|\g'|=|\g|+1 = k-r$;  
		\item Second,~$\deg P_{|\g|}(p)(p\s_3-1)\le [(k-1)-r]+1 = k-r$. 	
	\end{itemize}
\end{itemize}

The formula for~$R_k^a(x,p)$ may be verified in exactly the same way.  We write
\begin{align*}
R_k^a(x,p)
& = \frac{1}{(p^2-1)^{k+1}} \bigg[ \sum_{r=0}^{k-1} \sum_{\substack{\g\in \N^{2r+1}\\ |\g| = (k-1)-r}} U^2 Q_{\g}(x)P_{|\g|}(p) + i\p_x Q_{\g}(x) P_{|\g|}(p)\s_3 (p\s_3 + 1) \\
& \hspace{30 mm} - \sum_{r=1}^{k-1} \sum_{\substack{\eta\in \N^{2r}\\ |\eta| = (k-1)-r}} U \p_x Q_{\eta}(x) P_{|\eta|}(p) (p\s_3-1) \s_3 \bigg]. 
\end{align*}	
As above, we perform the routine verifications of the numerology as follows.
\begin{itemize}
\item When~$\g\in \N^{2r+1}$ and~$|\g| = (k-1)-r$, the term~$U^2 Q_\g(x)P_{|\g|}(x)$ can be absorbed into~$R^a_{k,r+1}(x,p)$:
\begin{itemize}
	\item First,~$U^2 Q_\g = Q_{(0,0,\g)}$, with~$(0,0,\g)\in \N^{2(r+1)+1}$,~$|(0,0,\g)| = |\g| = k-(r+1)$;
	\item Second,~$\deg P_{|\g|}(p) \le (k-1)-r = k-(r+1)$.  
\end{itemize}
\item When~$\g \in \N^{2r+1}$ and $|\g| = (k-1)-r$, the term~$i\p_x Q_\g(x)P_{|\g|}(p)\s_3(p\s_3+1)$ can be absorbed into~$R^a_{k,r}(x,p)$. When~$\eta\in \N^{2r}$ and $|\eta|=(k-1)-r$, the same is true of~$U\p_x Q_\eta(x)P_{|\eta|}(p)(p\s_3-~1)\s_3$.  
\begin{itemize}
	\item First, both~$\p_x Q_\g$ and~$U\p_x Q_\eta$ are sums of~$Q_{\g'}$'s, with~$\g'\in \N^{2r+1}$ and~$|\g'| = k-r$.  
	\item Second,~$P_{|\g|}(p)\s_3 (p\s_3 + 1)$ and~$P_{|\eta|}(p)(p\s_3-1)\s_3$ each have degree at most~$k-r$.  
\end{itemize}
\end{itemize}



\def\cprime{$'$}

\end{document}